\documentclass[reqno,12pt,letterpaper]{amsart}
\usepackage[proof]{zlmacros}
\usepackage[capitalize]{cleveref}

\title{Weyl Laws for Open Quantum Maps}
\author{Zhenhao Li}
\email{zhenhao@mit.edu}
\address{Department of Mathematics, Massachusetts Institute of Technology, Cambridge, MA 02139}

\begin{document}

\maketitle

\begin{abstract}
    We find Weyl upper bounds for the quantum open baker's map in the semiclassical limit. For the number of eigenvalues in an annulus, we derive the asymptotic upper bound $\mathcal O(N^\delta)$ where $\delta$ is the dimension of the trapped set of the baker's map and $(2 \pi N)^{-1}$ is the semiclassical parameter, which improves upon the previous result of $\mathcal O(N^{\delta + \epsilon})$. Furthermore, we derive a Weyl upper bound with explicit dependence on the inner radius of the annulus for quantum open baker's maps with Gevrey cutoffs. 
\end{abstract}

\section{Introduction}
Open quantum maps provide simple finite-dimensional models of open quantum chaos. This makes them especially conducive to numerical experimentation and thus appealing in the study of scattering resonances. They quantize a symplectic relation on a compact phase space. Such relations are toy models for Poincar\'e sections that arise when considering scattering Hamiltonians with hyperbolic trapped sets. See papers by Nonnenmacher--Sj\"ostrand--Zworski \cite{NSZ11, NSZ} for the precise description of the reduction from specific open quantum systems to open quantum maps using Poincar\'e sections. In this paper, the symplectic relation we consider is the classical baker's map on a 2-torus, which gives rise to the quantum open baker's map. We find a Weyl upper bound for the number of eigenvalues in an annulus.

The quantum open baker's map is an operator on 
\[\ell^2_N = \ell^2(\Z_N), \qquad \Z_N = \Z/(N\Z)\]
defined by the triple 
\begin{equation}\label{eq:triple}
    (M, \mathcal A, \chi), \quad M \in \N, \quad \mathcal A \subset \{0, \dots, M - 1\}, \quad \chi \in C_0^\infty((0, 1);[0, 1]).
\end{equation}
Here, $M$ is the \textit{base}, $\mathcal A$ is the \textit{alphabet}, and $\chi$ is the \textit{cutoff}. Put $N = K M$ where $K \in \mathcal \N$. Then the quantum open baker's map is given by 
\begin{equation*}
    B_N = \F^*_N \begin{pmatrix} \chi_{N/M} \F_{N/M} \chi_{N/M} & & \\ & \ddots & \\ & & \chi_{N/M} \F_{N/M} \chi_{N/M} \end{pmatrix}I_{\mathcal A, M}
\end{equation*}
where $\mathcal F_N$ is the unitary discrete Fourier transform, $I_{\mathcal A, M}$ is a diagonal matrix whose $(j, j)$-th entry is equal to one if $\lfloor j/K \rfloor \in \mathcal A$ and zero otherwise, and $\chi_{N/M}$ is a discretized smooth cutoff function. For example, for the triple $(3, \{0, 2\}, \chi)$ and $N = 3K$, the corresponding quantum open baker's map is then 
\[B_N = \F^*_N \begin{pmatrix} \chi_{K} \F_{K} \chi_{K} & 0 & 0 \\ 0 & 0 & 0 \\ 0 & 0 & \chi_{K} \F_{K} \chi_{K} \end{pmatrix}\]

Define the canonical relation on the torus $\mathbb{T}^2_{x, \xi}$ by 
\begin{equation}
    \begin{gathered}
    \varkappa_{M, \mathcal A} : (y, \eta) \mapsto (x, \xi) = \Big (My - a, \frac{\eta + a}{M} \Big), \\
    (y, \eta) \in \Big(\frac{a}{M}, \frac{a + 1}{M} \Big) \times (0, 1), \quad a \in \mathcal A. 
    \end{gathered}
\end{equation}
Then the corresponding semiclassical Fourier integral operator is given by 
\begin{equation*}
    \begin{gathered}
    \mathcal U_h := \sum_{a \in \mathcal A} \mathcal U_h^a \qquad \text{where} \\
    \mathcal U_h^a v(x) = \frac{M}{2 \pi h} \int_{\mathcal R^2} e^{\frac{i}{h}((x + a - My) \theta + xa/M)} \chi(M\theta) \chi(My - a)u(y)\, dy d\theta. 
    \end{gathered}
\end{equation*}
The quantum open baker's map can then be seen as the discrete analogue of this Fourier integral operator with the corresponding semiclassical parameter $(2 \pi N)^{-1}$. For a rigorous analysis of the analogy, see papers of Degli Esposti--Nonnenmacher--Winn \cite{DNW} and Nonnenmacher--Zworski \cite{NoZw07}. Heuristics can be found in earlier works of Bal\'azs--Voros \cite{BaVo} and Saraceno--Voros \cite{SaVo}. In view of this analogy, one would then expect that forward in time propagation by $B_N$ would lead to localization in frequency space to the Cantor set and backward propagation by $B_N$ would lead to localization in physical space to the Cantor set. Indeed, \cref{fig:propagation} demonstrates this property numerically. 

\begin{figure}
\begin{subfigure}{0.3\textwidth}
  \centering
  \includegraphics[width=\linewidth]{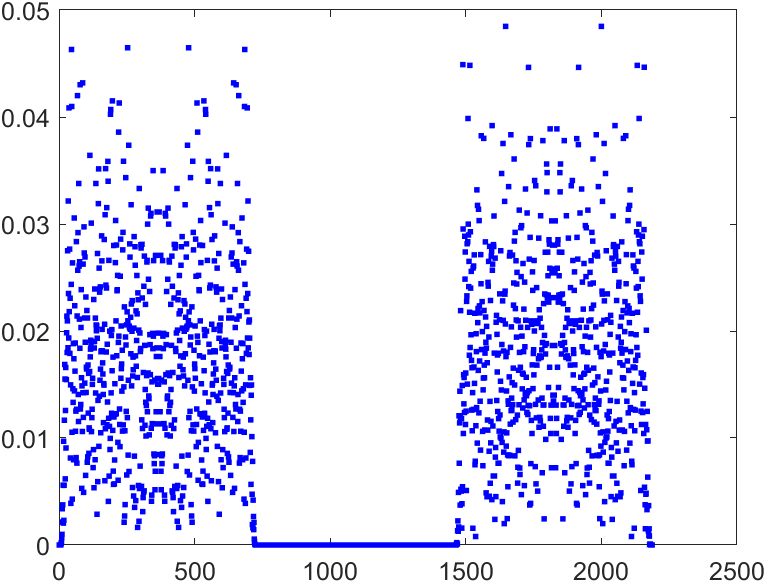}
  \caption{$\mathcal F_N B_N f$}
\end{subfigure}
\hfill
\begin{subfigure}{0.3\textwidth}
  \centering
  \includegraphics[width=\linewidth]{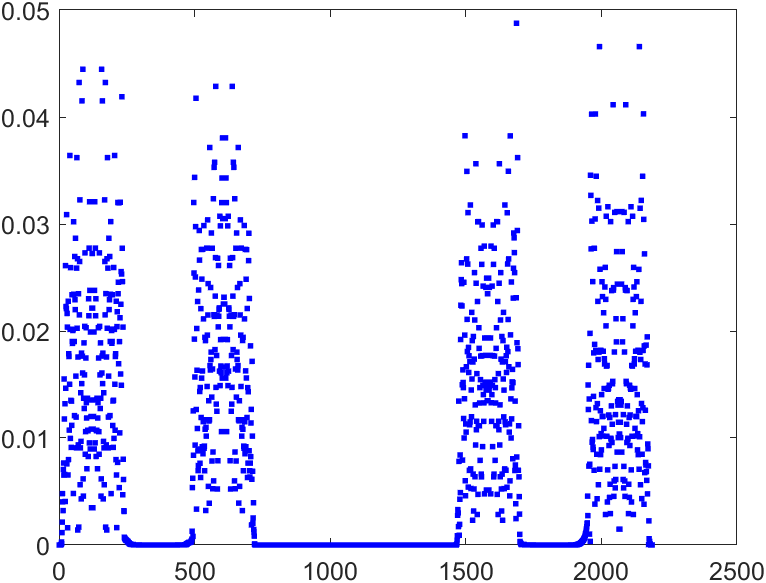}
  \caption{$\mathcal F_N B_N^2 f$}
\end{subfigure}
\hfill
\begin{subfigure}{0.3\textwidth}
  \centering
  \includegraphics[width=\linewidth]{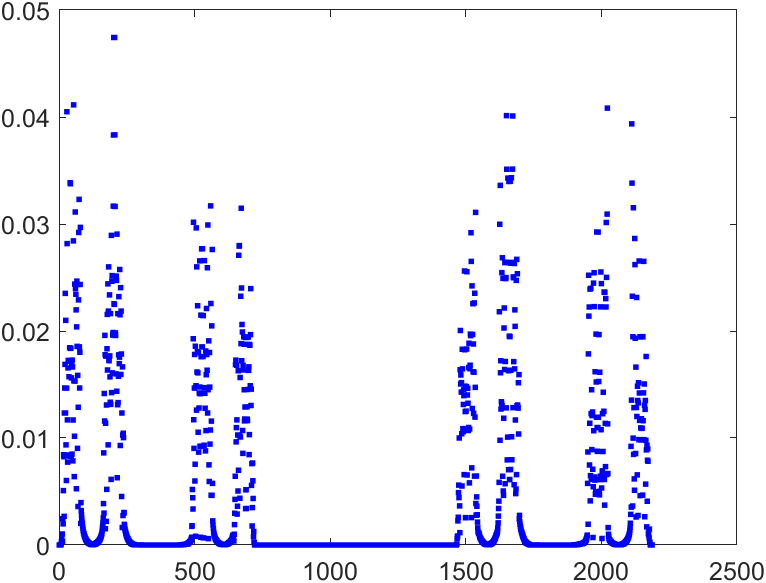}
  \caption{$\mathcal F_N B_N^3 f$}
\end{subfigure}

\begin{subfigure}{0.3\textwidth}
  \centering
  \includegraphics[width=\linewidth]{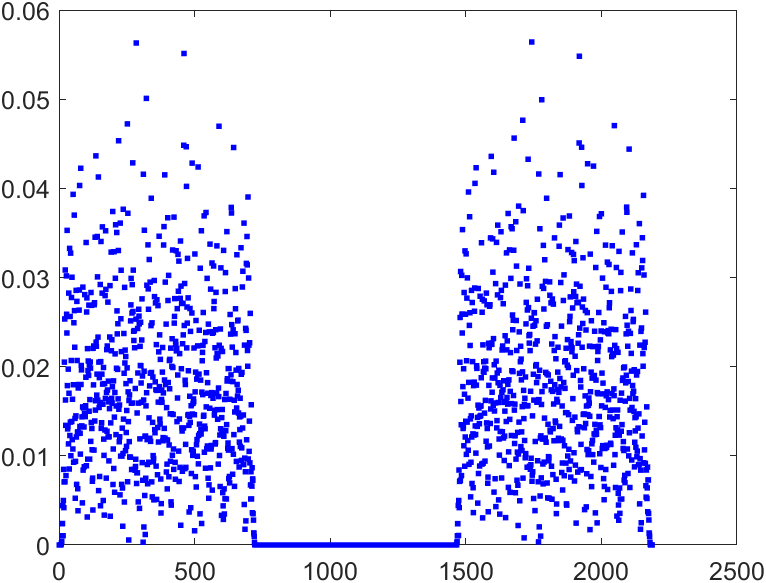}
  \caption{$(B_N^*) f$}
\end{subfigure}
\hfill
\begin{subfigure}{0.3\textwidth}
  \centering
  \includegraphics[width=\linewidth]{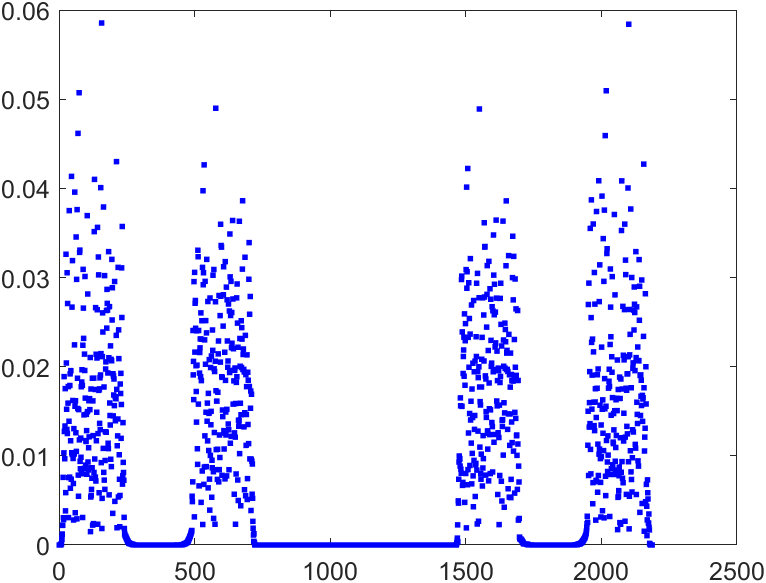}
  \caption{$(B_N^*)^2 f$}
\end{subfigure}
\hfill
\begin{subfigure}{0.3\textwidth}
  \centering
  \includegraphics[width=\linewidth]{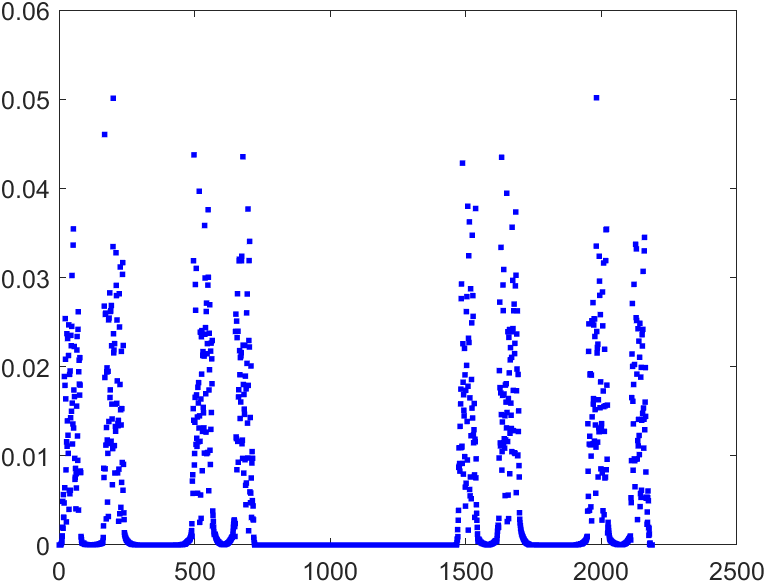}
  \caption{$(B_N^*)^3 f$}
\end{subfigure}
\caption{A demonstration of the localizing properties of $B_N$ for $M = 3$, $\mathcal A = \{0, 2\}$, and $N = 3^7$. An $\ell^2_N$ normalized vector $f$ was chosen uniformly at random. Plots (A) -- (C) are the frequency side of forward propagation, and (D) -- (F) are the spatial side of backward propagation. Each figure plots the absolute value of the indicated vector as a map from $\Z_N \to \R$}
\label{fig:propagation}
\end{figure}

Following the above observations, we should expect the eigenfunctions to be localized in frequency near the $(M, \mathcal A)$-Cantor set provided that the eigenvalues are not too small (see \cref{fig:typical_eigvec}). The maximum number of eigenfunctions that can be packed into such a region in phase space should then on the order of $N^\delta$ where
\begin{equation}
    \delta = \frac{\log |\mathcal A|}{\log M}
\end{equation}
is the dimension of the Cantor set. 
\begin{figure}
    \centering
    \includegraphics[width = 0.5\textwidth]{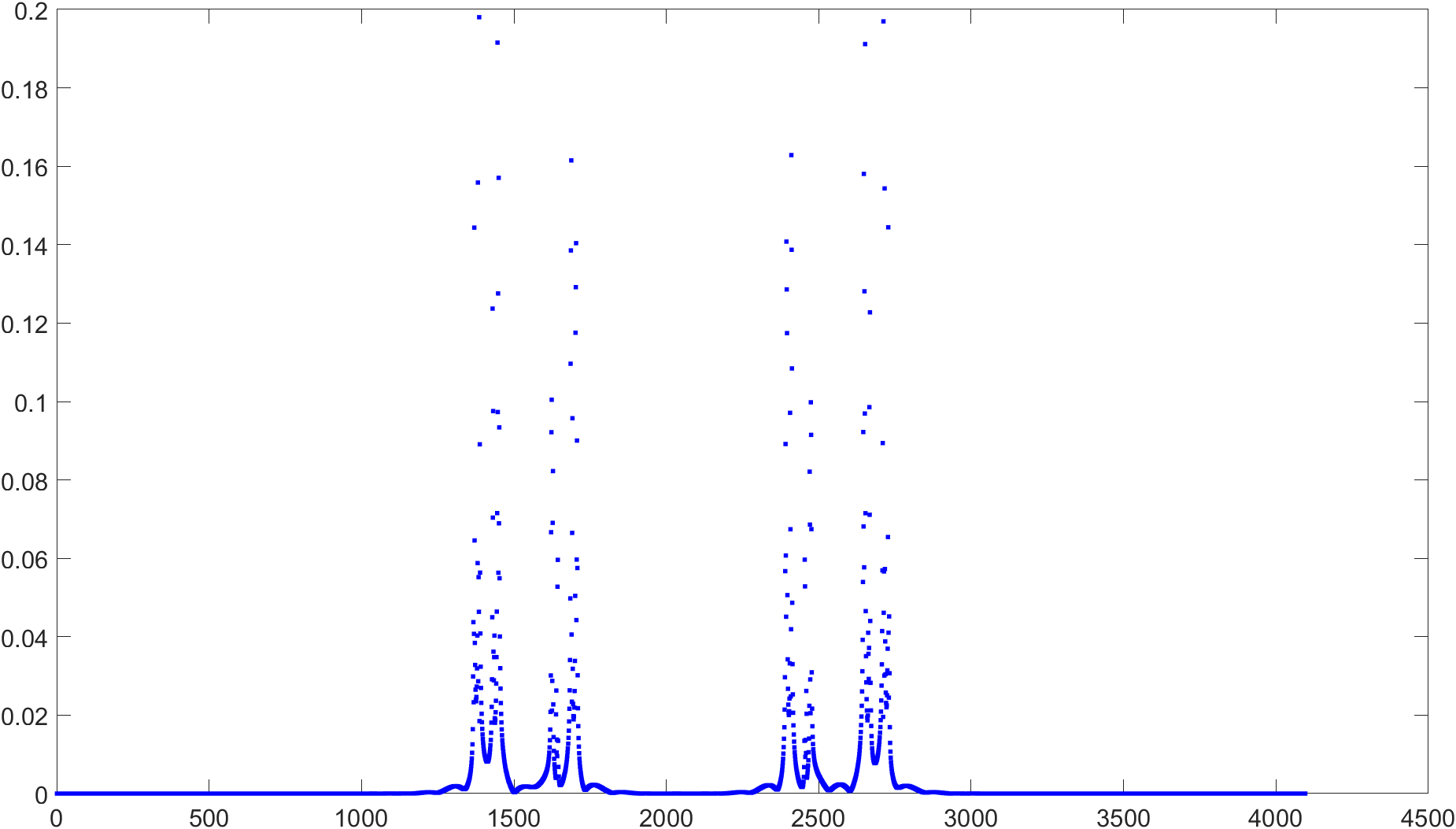}
    \caption{The Fourier side of a typical eigenvector. $M = 4$, $\mathcal A = \{1, 2\}$, and $\chi$ is identically $1$ on the Cantor set generated by $M$ and $\mathcal A$. Here, $N = 4^6$, and the absolute value of $\mathcal F_n v$ is plotted as a function from $\Z_N \to \R$ where $v$ is eigenvector with the 50th largest eigenvalue at $|\lambda| \approx M^{-0.4869}$.}
    \label{fig:typical_eigvec}
\end{figure}
Our result uses such localization properties to provide rigorous upper bounds to the number of eigenvalues of $B_N$ above a threshold. More specifically, consider the eigenvalue counting function 
\begin{equation}
    \mathcal N_N(\nu) = |\text{Spec}(B_N) \cap \{|\lambda| \ge M^{-\nu}\}|,
\end{equation}
defined for $\nu \ge 0$, where the eigenvalues are counted with multiplicities. Then we have the following Weyl upper bound:
\begin{theorem}\label{thm:weyl_law}
For each $\nu > 0$, we have as $N = K M \to \infty$, 
\begin{equation}
    \mathcal N_N(\nu) = \mathcal O(N^\delta).
\end{equation}
\end{theorem}

The proof of the theorem in \S \ref{sec:weyl_law_pf} follows the methods used in \cite{DZ}, in which the bound 
\[\mathcal N_N(\nu) = \mathcal O(N^{\delta + \epsilon})\]
for any $\epsilon > 0$ was proved. We obtain the $\epsilon$-improvement in this paper by using tighter propagation estimates and a modified approximate inverse identity (see \S 3.1 for details). 

With stronger assumptions on the decay of the cutoff function $\chi$, an explicit dependence of the upper bound on the depth of the spectrum $\nu$ can be extracted. In particular, we consider Gevrey class functions, first introduced in \cite{Gev} to study regularity of solutions to the heat equation. Given $s \ge 1$, a function $f \in \C^\infty(\R)$ is $s$-Gevrey if for every compact $K \subset \R$, there exists a constant $C_{K, f}$ such that 
\[\sup_{x \in K} |\partial^\alpha f(x)| \le C_{K, f}^{\alpha + 1} (\alpha!)^s\]
for all $\alpha \in \Z^+$. For $s = 1$, this is simply the space of real analytic functions, which cannot be compactly supported. However, for every $s > 1$, there exist smooth and compactly supported $s$-Gevrey functions. For $s > 1$, we write
\[\mathcal{G}^s_{\mathrm{c}}((0, 1)) = \{f \in C_c^\infty(\R):\, \text{$f$ is $s$-Gevrey and } \supp f \subset (0, 1)\}\]
Observe that if $\chi \in \mathcal{G}^s_{\mathrm{c}}((0, 1))$ for some $s > 1$, then 
\begin{equation}\label{eq:ftdecay}
    |\hat \chi(\xi)| \le Ce^{-c |\xi|^{1/s}},
\end{equation}
for some positive constants $C$ and $c$. Here, $\hat \chi$ denotes the usual Fourier transform given by 
\begin{equation}\label{eq:cft}
    \hat \chi(\epsilon) = \int_\R e^{-2 \pi i x \xi} \chi(x)\, dx.
\end{equation}
So even though we cannot have exponential decay of the Fourier transform that comes with analyticity, we can still get arbitrarily close. Finally, observe that for $\eta > 0$,
\begin{equation}\label{eq:ft_int}
    \int_\eta^\infty |\hat \chi(\xi)| \, d\xi \le Ce^{-\tilde c\eta^{1/s}}
\end{equation}
for some new constant $\tilde c < c$. See \cite[Chapter 1]{Ro} for a more detailed account on Gevrey classes and their Fourier decay properties. With this stronger cutoff decay assumption, we then have the following Weyl upper bound:

\begin{theorem}\label{thm:g_weyl_law}
Assume that $\chi \in \mathcal{G}^s_{\mathrm{c}}((0, 1))$ for some $s > 1$, then for all $\nu \ge 1$ and all sufficiently large $N = KM$ where $K \in \N$,  
\begin{equation}\label{eq:g_weyl_law}
    \mathcal N_N(\nu) \le CN^\delta \nu^{(1 - \delta)s}
\end{equation}
where the constant $C$ depend only on $\chi$ and $M$.
\end{theorem}

In the study of quantum chaos, open quantum systems given by the Laplacian on a noncompact Riemannian manifold whose geodesic flow is hyperbolic on the trapped set provide an important mathematical model. In the papers by Nonnenmacher--Sj\"ostrand--Zworski \cite{NSZ11, NSZ}, the study of resonances for such open quantum system is reduced to the open quantum map, so we should expect our results to run parallel to previous Weyl upper bounds for open quantum systems. We note that in the correspondence between quantum systems and quantum maps, if $\omega \in \C$ with $\Im \omega \le 0$ is a scattering resonance of the open quantum system, then 
\begin{equation}\label{eq:correspond}
    \lambda = e^{-i \omega \log M} = M^{-i \omega}
\end{equation}
is a corresponding eigenvalue of $B_N$, which makes sense in view of the fact that $B_N$ can be thought of as a toy model for the time $t = \log M$ propagator of an open quantum system with expansion rate 1. This means that Weyl upper bounds in horizontal strips below the real line should correspond to our Weyl law in an annulus. Weyl upper bounds for resonances of the Laplacian in strips (which corresponds to annuli for the open quantum maps by \cref{eq:correspond}) were first proved by Sj\"ostrand \cite{S}. This was done in the analytic category, which we cannot afford in our case since the cutoff $\chi$ is compactly supported. Using \cref{eq:correspond}, the corresponding bound that Sj\"ostrand found would give $\mathcal N_N(\nu) = \mathcal O(N^\delta \nu^{1 - \delta})$ where $\delta$ is the Minkowski dimension of the trapped set. Since we can only assume Gevrey for $s > 1$ in our setting, we see a corresponding loss in our result as we only have $\mathcal O(N^\delta \nu^{s(1 - \delta)})$. However, we remark that it appears from numerical experiments in \S \ref{numerics} that if the cutoff is identically $1$ near the trapped set, the Sj\"ostrand bound of $\mathcal O(N^\delta \nu^{1 - \delta})$ is recovered for $\nu$ not too large (\cref{fig:g_weyl_law}), but our methods do not appear to be able to account for this behavior. 

Weyl upper bounds for the Laplacian in fixed strips have been proved in various smooth settings by Guillop\'e--Lin--Zworski \cite{GLZ}, Zworski \cite{Zw99}, Sj\"ostrand--Zworski \cite{SjZw}, Nonnenmacher--Sj\"ostrand--Zworski \cite{NSZ11, NSZ}, and Datchev--Dyatlov \cite{DaDy}. These give the corresponding bounds $\mathcal N_N(\nu) = \mathcal O(N^\delta)$, which aligns with our result in \cref{thm:weyl_law}. Physical microwave experiments on the Weyl law asymptotics have been done by Potzuweit et al. \cite{PWBKSZ}, and various numerical experiments can be found in, Lu--Sridhar--Zworski \cite{LSZ}, Borthwick--Weich \cite{BoWe}, Borthwick \cite{Bor}, and Borthwick--Dyatlov--Weich \cite{Dy15b}. The main idea behind deriving the Weyl upper bounds is the localization of the eigenfunctions in phase space (see \cref{prop:long_time}), and this was observed numerically by Keating et al. \cite{KNPS}. In the seeting of Walsh quantization, which uses a modified Fourier transform, the localization was later proved by Keating et al. \cite{KNNS}

\section{Open quantum maps}
In this section, we establish some basic definitions and a general nonstationary phase estimate. We then give a more detailed definition of the quantum open baker's map, and use the nonstationary phase estimate to prove the one-step propagation estimate for the quantum open baker's map $B_N$ that will be iterated in order to get propogation of singularities estimates for long times.  

\subsection{Preliminaries}
For $N \in \N$, we have the abelian group 
\[\Z_N := \Z/N\Z \simeq \{0, \dots, N - 1\},\]
and we have the associated $\ell^2_N = \ell^2(\Z_N)$ space of functions $u: \Z_N \to \C$ equipped with the norm 
\[\|u \|_{\ell^2_N}^2 = \sum_{j = 0}^{N - 1} |u(j)|^2.\]
The unitary Fourier transform on $\ell^2_N$ is given by 
\[\mathcal F_Nu(j) = \frac{1}{\sqrt{N}} \sum_{\ell = 0}^{N - 1} \exp\Big(-\frac{2 \pi i j \ell}{N} \Big) u(\ell).\]
Given a function $\varphi:[0, 1] \to \C$, its discretization is a function denoted $\varphi_N \in \ell^2_N$ given by 
\begin{equation}
    \varphi_N(j) = \varphi\Big( \frac{j}{N} \Big), \quad j \in {0, \dots, N - 1}.
\end{equation}
We denote the corresponding Fourier multiplier by
\begin{equation}\label{eq:fourier_mult}
    \varphi_N^\F = \F_N^* \varphi_N \F_N
\end{equation}
For the distance function on $[0, 1]$, we consider the interval with $0$ and $1$ identified. In particular, 
\[d(x, y) = \min\{|x - y|, 1 - |x - y|\}.\]
For the distance between sets, we have the usual
\[d(U, V) = \inf_{x \in U,\, y \in V} d(x, y).\]

Now we have the following nonstationary phase estimate. 
\begin{lemma}\label{lem:nonstationary}
	Fix $\chi \in C_c^\infty((0, 1))$. Assume that $a \in \Z_N$ and
	\[d \left( \frac{a}{N}, 0 \right) \ge r\]
	for some $r \in (0, 1/2)$. Then
	\begin{equation}\label{eq:nonstationaryphase}
	    \left|\sum_{m = 0}^{N - 1} \exp \left( \frac{2 \pi i am}{N} \right) \chi \left(\frac{m}{N} \right) \right| \le N \cdot g_\chi(Nr)
	\end{equation}
	where 
	\begin{equation}
	    g_\chi(x) \le \begin{cases} C_nx^{-n} \quad \forall n & \text{if } \varphi \in C_c^\infty((0, 1)) \\ Ce^{-cx^{1/s}} & \text{if } \varphi \in \mathcal{G}^s_{\mathrm{c}}((0, 1)) \end{cases}
	\end{equation}
	The positive constants $C_n$, $C$, and $c$ depend only on the choice of $\chi$. 
\end{lemma}

\begin{proof}
	The Poisson summation formula gives
	\begin{equation*}
	    \sum_{m = 0}^{N - 1} \exp\left( \frac{2 \pi i am}{N} \right) \chi \left( \frac{m}{N} \right) = N \sum_{\ell \in \Z} \hat \chi(N \ell - a).
	\end{equation*}
	Note that $\hat \chi$ is rapidly decaying and by our assumption, $N\ell - a > Nr$. Therefore, for every $n \ge 0$,
	\begin{align*}
	    \Big | N\sum_{\ell \in \Z} \hat \chi(N \ell - a) \Big| &\le C N \sum_{\ell \in \Z} (N \ell - a)^{-n} \\
	    &\le 2CN \sum_{\ell \ge 0} (N(r + \ell)) ^{-n} \, \\
	    &\le C_n N(Nr)^{-n}
	\end{align*}
	where the constant depends only on $n$. Similarly, if $\chi \in \mathcal{G}^s_{\mathrm{c}}((0, 1))$ for some $s > 1$, then in view of \cref{eq:ftdecay} and \cref{eq:ft_int},  
	\begin{align*}
	    \Big|N\sum_{\ell \in \Z} \hat \chi(N \ell - a) \Big| &\le C N \sum_{\ell \in \Z} \exp\left(-c|N \ell - a|^{\frac{1}{s}} \right) \\
	    &\le 2CN \cdot \exp\left(-c(Nr)^{\frac{1}{s}} \right) + 2 C \int_{Nr}^\infty \exp\left(-c \xi^{\frac{1}{s}} \right) \, d \xi \\
	    &\le \tilde C N\exp\left(-\tilde c (Nr)^{\frac{1}{s}}\right),
	\end{align*}
	where we interpret the sum as a lower Riemann sum to bound by the integral. Again, all the constants above depend only on $\chi$. Therefore we have both of the desired estimates. 
\end{proof}

\subsection{One-step propagation}
Let the triple $(M, \mathcal A, \chi)$ be as in \cref{eq:triple}, and put $N = K M$ where $K \in \mathcal \N$. Define the projection $\Pi_a: \ell^2_N \to \ell^2_{N/M}$, $a \in \{0, \dots, M - 1\}$, by
\begin{equation}
    \Pi_a u(j) = u \Big(j + a K \Big), \quad u \in \ell^2_N, j \in \Big\{0, \dots, K - 1 \Big\}. 
\end{equation}
Then the open quantum map $B_N: \ell^2_N \to \ell^2_N$ can be written as 
\begin{equation*}
    B_N = \sum_{a \in \mathcal A} \mathcal F^*_N \Pi_a^* \chi_{N/M} \mathcal F_{N/M} \chi_{N/M} \Pi_a.
\end{equation*}
Expanding out the Fourier transforms, we have the formula 
\begin{multline}\label{eq:full_expansion}
    B_Nu(j) = \frac{\sqrt{M}}{N} \sum_{a \in \mathcal A} \sum_{m, \ell = 0}^{K - 1} \exp\left[2 \pi i \left(\frac{(j - M\ell)m}{N} + \frac{ja}{M} \right)\right] \\ \chi \Big(\frac{m}{K} \Big) \chi \Big(\frac{\ell}{K} \Big)u(\ell + aK). 
\end{multline}
It will be useful in the propagation estimates to define the expanding map
\[\Phi = \Phi_{M, \mathcal A}: \bigsqcup_{a \in \mathcal A} \Big(\frac{a}{M}, \frac{a + 1}{M} \Big) \to (0, 1)\]
given by 
\begin{equation}\label{eq:expand_map}
    \Phi(x) = Mx - a, \quad x \in \Big(\frac{a}{M}, \frac{a + 1}{M} \Big)
\end{equation}

We will obtain estimates in terms of the constants $C_n$ for the propagation of singularities. In particular, we start by showing that by applying $B_N$ once, the resulting function will be roughly microlocalized to $\varkappa_{M, \mathcal A}((0, 1)^2)$, and by applying $B_N^*$ once, the resulting function will be roughly microlocalized to $\varkappa_{M, \mathcal A}^{-1} ((0, 1)^2)$ (of course, this is imprecise since our setting is discrete). This localization behavior is clear in the classical open baker's map $\varkappa_{M, \mathcal A}$. In the discrete setting, it is then natural to consider $B_N$ as a matrix consisting of blocks that reflect the classical structure of the baker's map, and each block will be rapidly decaying away from the diagonal, so then we can apply Schur's bound to control the norm. To make precise the above heuristics, we have the following estimate.

\begin{proposition}\label{prop:gap_est}
	Assume that $\varphi, \psi:[0, 1] \to [0, 1]$ such that 
	\begin{equation}\label{eq:propogationspacing}
	    d(\supp \psi, \Phi^{-1}(\supp \varphi)) \ge r
	\end{equation}
	where $\Phi$ is the expanding map as defined in \cref{eq:expand_map} and $r$ is a small gap satisfying
	\begin{equation}\label{eq:r_cond}
	    0 < Mr \le 2d(\supp \chi, 0).
	\end{equation}
	Let $\psi_N^\F$ and $\varphi_N^\F$ are Fourier multipliers as defined in \textup{(\ref{eq:fourier_mult})}. Then
	\begin{gather*}
		\|\varphi_N B_N \psi_N \|_{\ell_N^2 \to \ell_N^2} \le \tilde g_\chi(Nr) \\
		\|\psi_N^{\F} B_N \varphi_N^\F \|_{\ell_N^2 \to \ell_N^2} \le \tilde g_\chi(Nr)
	\end{gather*}
	where
	\begin{equation}\label{eq:tilde_g}
	    \tilde g_\chi(x) \le \begin{cases} C_n x^{-n} \quad \forall n & \text{if } \varphi \in C_c^\infty((0, 1)) \\ Ce^{-cx^{1/s}} & \text{if } \varphi \in \mathcal{G}^s_{\mathrm{c}}((0, 1)) \end{cases}
	\end{equation} 
	where $C_n$, $c$, and $C$ are positive constants dependending only on $\chi$. 
\end{proposition}
\begin{proof}
	1. We computed each entry of $\varphi_N B_N \psi_N$ as an $N \times N$ matrix, $N = KM$ for some $K \in \N$. From the expansion \cref{eq:full_expansion} for $B_N$, we can write 
	\begin{equation*}
	    \varphi_N B_N \psi_N u(j) = \sum_{a \in \mathcal A} \sum_{\ell = 0}^{K - 1} A_{j\ell}^a u(\ell + aK)
	\end{equation*}
	where
	\begin{align*}
	    A_{j\ell}^a &= \frac{\sqrt{M}}{N} \varphi \Big( \frac{j}{N} \Big) \exp \Big( \frac{2 \pi i aj}{M} \Big) \chi\Big(\frac{\ell}{K} \Big) \psi\Big( \frac{\ell}{N} + \frac{a}{M} \Big) \widetilde A_{j\ell} \\
	    \widetilde A_{j\ell} &= \sum_{m = 0}^{K - 1} \exp \Big( \frac{2 \pi i m(j - \ell M)}{N} \Big)\chi \Big(\frac{m}{K} \Big)
	\end{align*}
	Observe that $A^a_{j\ell}$ can be nonzero only when 
	\begin{equation}\label{eq:awayfromdiag}
	    a \in \mathcal A, \qquad \frac{j}{N} \in \supp \varphi, \qquad \frac{\ell}{N} + \frac{a}{M} \in \supp \psi, \qquad \frac{\ell}{K} \in \supp \chi.
	\end{equation}
	For $j$ and $\ell$ such that condition (\ref{eq:awayfromdiag}) holds, it follows from \cref{eq:propogationspacing} that
	\begin{equation}\label{eq:lowerdiagbound}
	    d \Big(\frac{j - \ell M}{N}, 0 \Big) \ge \min\{ M\cdot r, 2d(\supp \chi, 0)\} \ge M \cdot r.
	\end{equation}
	Here, note that we crucially used condition (\ref{eq:r_cond}) on $r$, which controls the case that $\supp \psi$ contains a neighborhood of $ak/M$ and $\supp \Phi^{-1}(\varphi)$ contains a neighborhood of $(a+1)k/M$ or vice versa for some $a \in \mathcal A$. 
	
	\noindent
	2. We now use Schur's bound (see for instance \cite[\S 4.5.1]{Zw}) to bound the operator of $\varphi_N B_N \psi_N$. In particular, it suffices to show that 
	\begin{gather*}
	    \max_{0 \le j \le N - 1} \sum_{a \in \mathcal A} \sum_{\ell = 0}^{K - 1} |A_{j\ell}^a| \le \tilde g_\chi(Nr) \\
	    \max_{\substack{a \in \mathcal A \\ 0 \le \ell \le N/M - 1}} \sum_{j = 0}^{N - 1} |A_{j\ell}^a| \le \tilde g_\chi(Nr)
	\end{gather*}
	for some $\tilde g_\chi$ that satisfies (\ref{eq:tilde_g}) in order to conclude $\|\varphi_N B_N \psi_N\|_{\ell^2_N \to \ell^2_N} \le \tilde g_\chi(Nr)$. 
	
	Let $g = g_{\chi_M}$ as in \cref{lem:nonstationary} where $\chi_M(x) = \chi(Mx)$. Then for any $a \in \mathcal A$ and $j \in \{0, \dots, N - 1\}$, (\ref{eq:awayfromdiag}) and (\ref{eq:lowerdiagbound}) means that the conditions of \cref{lem:nonstationary} are satisfied, so
	\begin{align}
	    \sum_{\ell = 0}^{K - 1} |A_{j \ell}^a| &\le \frac{\sqrt{M}}{N} \sum_{\ell = 0}^{K - 1} |\widetilde A_{j \ell} | \nonumber \\
	    &\le \sqrt{M} \sum_{\ell : \, d(\frac{j - \ell M}{N}, 0) \ge Mr} g(j - \ell M) \nonumber \\
	    &\le 2\sqrt{M} \sum_{\ell \ge Nr} g(\ell M) \label{eq:schurhorizontal}
	\end{align}
    Similarly, for any $\ell$ and $a$, we have that
	\begin{align}
	    \sum_{j = 0}^{N - 1} |A^a_{j \ell}| &\le \frac{\sqrt{M}}{N} \sum_{j = 0}^{N - 1} |\widetilde A_{j \ell}| \nonumber \\
	    &\le \sqrt{M} \sum_{j : \, d(\frac{j - \ell M}{N}, 0) \ge Mr} g(j - \ell M) \nonumber \\
	    &\le 2M^{3/2} \sum_{j \ge NMr} g(j) \label{eq:schurvertical}
	\end{align}

	\noindent
	3. Now we substitute in the relevant $g_\chi$ into the bounds from Step 2 to recover the desired estimates. With no extra assumptions on $\chi$, $g$ decays rapidly. Then it follows from \cref{eq:schurhorizontal} and \cref{eq:schurvertical}
	\begin{equation}\label{eq:schur}
	    \begin{gathered}
	    \sum_{a \in \mathcal A} \sum_{\ell = 0}^{K - 1} |A^a_{j \ell}| \le C_n \sum_{\ell \ge Nr}(\ell M)^{-n - 1} \le C_n (Nr)^{-1} \\
	    \sum_{j = 0}^{N - 1} |A_{j\ell}^a| \le C_n \sum_{j \ge NMr} g(j) \le C_n (Nr)^{-1}
	    \end{gathered}
	\end{equation}
	where $C_n$ can change from line to line but depends only on $n$, $\chi$, and $M$. On the other hand, if $\chi \in \mathcal{G}^s_{\mathrm{c}}((0, 1))$, then $g \le C\exp(-cx^{\frac{1}{s}})$. Therefore, it follows from \cref{eq:schurhorizontal} and \cref{eq:schurvertical} that 
	\begin{equation}\label{eq:g_schur}
	    \begin{gathered}
	    \sum_{a \in \mathcal A} \sum_{\ell = 0}^{K - 1} |A^a_{j \ell}| \le C \int_{Nr}^\infty \exp \left(-cx^{\frac{1}{s}}\right) \, dx \le \tilde C \exp\left(-\tilde cx^{\frac{1}{s}} \right) \\
	    \sum_{j = 0}^{N - 1} |A^a_{j \ell}| \le C \int_{NMr}^\infty \exp \left(-cx^{\frac{1}{s}}\right) \, dx \le \tilde C\exp\left(- \tilde cx^{\frac{1}{s}}\right)
	    \end{gathered}
	\end{equation}
	where the constants depend only on $\chi$ and $M$. Here, the sum can be seen as a lower Riemann sum, which is bounded by the corresponding integral. Therefore, by Schur's estimate, \cref{eq:schur} and \cref{eq:g_schur} yield the desired estimates on $\varphi_N B_N \psi_N$. To obtain the estimates on the Fourier side, we simply have 
	\begin{equation*}
	    \| \psi_N^\F B_N \varphi_N^\F\|_{\ell^2_N \to \ell^2_N} = \|\F_N^* (\overline{\varphi_N B_N \psi_N})^* \F_N\|_{\ell^2_N \to \ell^2_N} \le \|\varphi_N B_N \psi_N\|_{\ell^2_N \to \ell^2_N},
	\end{equation*}
	which gives the identical bounds for the Fourier side. 
\end{proof}

The manifestation of this propagation estimate is clear in \cref{fig:propagation}. Each time a random function over $\Z_N$ is propagated by $B_N$, it localizes in frequency space to the next Cantor subset, and similarly propagation by $B_N^*$ yields localization in physical space to the next Cantor subset. 

\section{Propagation of singularities}
Now we are in a position to iteratively apply the one-step propagation estimate \cref{prop:gap_est} to obtain bounds on propagation for a long time. First, we will derive a general estimate for long time propagation. The general estimate will then be applied to the case $\chi \in C_c^\infty((0, 1))$, and then to the Gevrey case $\chi \in \mathcal{G}^s_{\mathrm{c}}((0, 1))$ for $s > 1$. 
\subsection{Long-time propagation}
\label{sec:no_decay}
Let $N = K M$ for $K \in \N$ and let $\Phi$ denote the expanding map as defined in the (\ref{eq:expand_map}). Define the fattened Cantor set 
\begin{equation}\label{eq:fat_Cantor}
    X_j := \{\Phi^{-j}(x) + y \mod 1:\, x \in [0, 1], |y| \le a_j\} = \Phi^{-j}([0, 1]) + [-a_j, a_j].
\end{equation}
The gap $a_j$ will be adjusted later. For now, we only need to assume that $a_j > a_{j - 1}/M$. Set 
\begin{equation}\label{eq:localizer}
    A_j = (\indic_{X_j})_N^{\F}.
\end{equation}
Roughly speaking, $A_j$ is a localizing operator on the Fourier side that localizes $N \cdot a_j$-close to the $j$-th discrete Cantor subset in $\Z_N$. We remark that the discrete Cantor subsets of $\Z_N$ are generally not defined in our setting since we do not assume that $N$ is a power of $M$. We only assume that $N$ is a multiple of $M$ in order to ensure that $B_N$ is well-defined. However, this is not a problem since the fattened Cantor sets are simply defined on the continuum and then discretized. 

Note that 
\[d(\Phi^{-1}(X_{j - 1}), [0, 1] \setminus X_j) \le a_j - \frac{a_{j - 1}}{M}.\] 
Define the gap distance by $d_1 = a_1$ and 
\begin{equation}\label{eq:distance}
    d_j = N \cdot \left( a_{j} - \frac{a_{j-1}}{M} \right)
\end{equation}
for $j \ge 2$. Therefore, by \cref{prop:gap_est}, we have estimates of the form 
\begin{equation}\label{localization}
	(1 - A_j) B_N A_{j - 1} = R_j
\end{equation}
where
\[\|R_j\|_{\ell^2_N \to \ell^2_N} \le \tilde g_\chi \left(d_j \right).\]
provided that condition (\ref{eq:r_cond}) holds, i.e.
\begin{equation}\label{r_cond'}
    0 < \frac{d_j}{N} \le \frac{2}{M} d(\supp \chi, 0) 
\end{equation}
We propagate the estimate (\ref{localization}) to obtain long time estimates in the following proposition.

Define the annular domain 
\begin{equation}\label{eq:Omega}
    \Omega_\nu:= \{M^{-\nu} < |\lambda| < 5 \} \subset \C.
\end{equation}

\begin{proposition}\label{prop:long_time}
Let $N = K M$ for some $K \in \N$. Fix a sequence 
\[\{d_j\}_{j = 1}^\ell, \qquad \ell \le \frac{\log N}{\log M}\] 
such that the condition \textup{(\ref{r_cond'})} holds. Then there exists a Fourier multiplier 
\[A:\ell^2_N \to \ell^2_N\]
and families of operators 
\[Z(\lambda):\ell^2_N \to \ell^2_N \qquad \mathcal R(\lambda):\ell^2_N \to \ell^2_N\]
that satisfy the identity 
    \begin{equation}\label{eq:long-time-prop}
        I = Z(\lambda) (B_N - \lambda) + \mathcal R(\lambda) + A
    \end{equation}
such that 
\begin{enumerate}\label{eq:remainder_est}
    \item we have the remainder estimate 
    \begin{equation}\label{eq:mathcal_R}
        \|\mathcal R(\lambda)\|_{\ell^2_N \to \ell^2_N} \le \sum_{j = 0}^{\ell - 1} |\lambda|^{-j - 1} \tilde g_\chi(d_{\ell - j})
    \end{equation}
    where $\tilde g_\chi$ is the same as in \cref{eq:tilde_g}. 
    \item $A$ has rank bounded by 
    \begin{equation}\label{rank_est}
        \rank A \le 2M^{\ell \delta} \left[ \frac{N}{M^\ell} + 2\sum_{k = 1}^\ell \frac{d_k}{M^{\ell - k}} \right]
    \end{equation}
\end{enumerate}

\end{proposition}

\begin{proof}
We obtain the identity from iterating propagation estimate \cref{localization}. Put
\begin{equation}\label{eq:a_j}
    a_j = \frac{1}{N} \sum_{k = 1}^j \frac{d_k}{M^{j - k}}
\end{equation}
so that
\[d_j = N \cdot \Big (a_j - \frac{a_{j - 1}}{M} \Big).\]
Then we can form the fattened Cantor sets $X_j$ as in \cref{eq:fat_Cantor} with the corresponding Fourier localizers $A_j$ defined in \cref{eq:localizer}. Iterating the estimate \cref{localization} $\ell$-times,
we find
\begin{align*}
    (1 - A_\ell) B_N^\ell =& (1 - A_\ell) B_N A_{\ell - 1} B_N^{\ell - 1} + (1 - A_\ell)B_N (1 - A_{\ell - 1}) B_N^{\ell - 1} \\
    =&\sum_{j = 0}^{\ell - 1} (1 - A_\ell)B_N(1 - A_{\ell - 1}) B_N \dots (1 - A_{\ell - j}) B_N A_{\ell - j - 1} B_N^{\ell - j - 1} \\
    =& E_\ell(B_N - \lambda) \\
    & \qquad + \sum_{j = 0}^{\ell - 1} \lambda^{\ell - j - 1} (1 - A_\ell)B_N(1 - A_{\ell - 1}) B_N \dots (1 - A_{\ell - j}) B_N A_{\ell - j - 1}
\end{align*}
where
\[E_\ell = \sum_{j = 0}^{\ell - 1} \sum_{k = 0}^{\ell - j - 2} (1 - A_\ell)B_N(1 - A_{\ell - 1}) B_N \dots (1 - A_{\ell - j}) B_N A_{\ell - j - 1}(\lambda^k B^{\ell - j - 2 - k}). \]
Now, we have the desired approximate inverse identity given by 
\begin{align}
    I &= -\left(\sum_{0 \le k < \ell} \lambda^{-1-k} (I - A_\ell)(B_N)^k \right)(B_N - \lambda) + \lambda^{-\ell} (1 - A_\ell)(B_N)^\ell + A_\ell \nonumber\\
    &= -\left(\sum_{0 \le k < \ell} \lambda^{-1-k} (I - A_\ell)(B_N)^k \right)(B_N - \lambda) + \lambda^{-\ell} E_\ell(B - \lambda) \nonumber\\
    & \quad + \lambda^{-\ell}\sum_{j = 0}^{\ell - 1} \lambda^{\ell - j - 1} (1 - A_\ell)B_N(1 - A_{\ell - 1}) B_N \dots (1 - A_{\ell - j}) B_N A_{\ell - j - 1} + A_\ell \nonumber \\
    &= Z_\ell (B_N - \lambda) + \mathcal R_\ell + A_\ell \label{eq:approximate_id}
\end{align}
where
\begin{gather*}
    Z_\ell = -\left(\sum_{0 \le k < \ell} \lambda^{-1-k} (I - A_\ell)(B_N)^k \right) + \lambda^\ell E_\ell\\ 
    \mathcal R_\ell = \sum_{j = 0}^{\ell - 1} \lambda^{- j - 1} B_N^j (1 - A_{\ell - j} )B_N A_{\ell - j - 1}
\end{gather*}
By assumption, condition (\ref{r_cond'}) is satisfied, so \cref{localization} gives the desired remainder bound 
\begin{equation*}
    \| \mathcal R_\ell\|_{\ell^2_N \to \ell^2_N} \le \sum_{j = 0}^{\ell - 1} |\lambda|^{ - j - 1} \tilde g\left(N \cdot \left( a_{\ell - j} - \frac{a_{\ell - j - 1}}{M} \right) \right).
\end{equation*}

To bound the rank of $A_\ell$, we observe that $\Phi^{-\ell}([0, 1])$ is the union of $M^{\delta \ell}$ copies of intervals of length $M^{-\ell}$. Then from (\ref{eq:fat_Cantor}), the measure of $X_\ell$ can be bounded by 
\[|X_\ell| \le M^{\delta \ell}(M^{-\ell} + 2a_\ell).\]
By \cref{eq:a_j}, we then obtain the desired bound
\[\rank A_\ell \le 2N \cdot |X_\ell| \le 2M^{\ell \delta} \left[ \frac{N}{M^\ell} + 2 \sum_{k = 1}^\ell \frac{d_k}{M^{\ell - k}} \right].\]
Note that the above inequality holds since $M^{-\ell} \le 1/N$ by assumption, and the factor of 2 here is merely to account for the discretization. The proposition then follows by putting 
\begin{gather*}
    A = A_\ell \\
    Z(\lambda) = Z_\ell \\
    \mathcal R(\lambda) = \mathcal R_\ell
\end{gather*}
\end{proof}

\subsection{Propagation with smooth cutoff}\label{sec:smooth_gaps}
Ultimately, we want to find the asymptotics of the eigenvalue counting function as $N \to \infty$ for a fixed $\nu > 0$. Therefore, we need some uniform control over the identity \cref{eq:long-time-prop} for all sufficiently large $N$ at a fixed $\nu$. In particular, we choose $d_j$ so that $\mathcal R_\ell$ will be uniformly small for all large $N$ and the rank of $A$ will be on the order $N^\delta$. 

In the case that $\chi \in C_c^\infty((0, 1)$, recall that $\tilde g = \tilde g_\chi$ is rapidly decaying. Then provided that the gaps $d_{\ell - j}$ are chosen so that (\ref{r_cond'}) holds, the identity (\ref{eq:long-time-prop}) holds with the remainder estimate (\ref{eq:mathcal_R}) given by 
\begin{equation}\label{eq:remainderbound}
\|\mathcal R_{\ell}\|_{\ell^2_N \to \ell^2_N} \le C_n \sum_{j = 0}^{\ell - 1} \lambda^{-j - 1} d_{\ell - j}^n
\end{equation}
for constants $C_n$ depending only on $\chi$. Note that for $|\lambda| < 1$, the factor $\lambda^{-j - 1}$ in \cref{eq:remainderbound} increases exponentially as $j$ increases. In order for $\mathcal R_{\ell}$ to be small in norm, this growth needs to be tempered by $d_{\ell - j}$. The strategy is to choose $d_{\ell -  j}$ in such a way so that the sum in \cref{eq:remainderbound} becomes exponentially decreasing in $j$.

Therefore, we put
\begin{equation}\label{eq:gap_dist_choice}
    d_{\ell - j} = \frac{L \cdot M^j}{1.5^j}.
\end{equation}
We will choose $L > 0$. Meanwhile, let the time of propagation be
\begin{equation}
    \ell = \left \lfloor \frac{\log N}{\log M} \right \rfloor.
\end{equation}
Since $M \ge 2$, there exists a sufficiently large $n$ so that
\begin{equation}\label{eq:n_condition}
    \frac{1.5^n}{M^{n - \nu}} < \frac{1}{2}.
\end{equation}
Then choose $L$ so that 
\begin{equation}\label{eq:L_condition}
    L^n > 4 M^{\nu} C_n
\end{equation}
where the constant $C_n$ is as in \cref{eq:remainderbound}. Note that the choice of $L$ depends only on $\chi$ and $M$. Next,
\[\frac{d_{\ell - j}}{N} \le \frac{L}{1.5^j} \frac{M^j}{N} \le \frac{L}{1.5^\ell}.\]
Therefore for all sufficiently large $\ell$ (and thus for all sufficiently large $N$), condition (\ref{r_cond'}) will be satisfied. Therefore, \cref{prop:long_time} applies and we have the remainder estimate
\begin{align}
    \|\mathcal R_\ell \|_{\ell_N^2 \to \ell_N^2} &\le C_n \sum_{j = 0}^{\ell - 1} \lambda^{-j - 1}d_{\ell - j}^{-n} \nonumber \\
    &\le \frac{C_n M^{\nu}}{L^n} \sum_{j = 0}^{\ell - 1} \left( \frac{2^{n}}{M^{(n - \nu)}} \right)^j \nonumber \\
    &\le \frac{1}{2} \label{eq:small_remainder}
\end{align}
Furthermore, \cref{rank_est} gives the rank bound 
\[\rank A \le 2N^\delta \left(1 + 2 \sum_{k = 1}^\ell \frac{L}{1.5^{\ell - k}} \right) \le C N^\delta\]
where $C$ depends only on $\nu$, $\chi$ and $M$. 
In summary, we have the following corollary of \cref{prop:long_time}:
\begin{corollary}\label{cor:long_time}
Consider the quantum open maps given by the triple $(M, \mathcal A, \chi)$ and fix $\nu > 0$. For all sufficiently large $N = K M$ where $K \in \N$ there exists operators $A$, $Z(\lambda)$, and $\mathcal R(\lambda)$ on $\ell^2_N$ as in \cref{prop:long_time} that satisfies the identity \cref{eq:long-time-prop}. Furthermore, they satisfy the remainder estimate 
\[\|\mathcal R(\lambda)\|_{\ell^2_N \to \ell^2_N} \le 1/2\]
and the rank bound 
\[\rank A \le CN^\delta\]
where $C$ does not depend on $N$. 
\end{corollary}

\subsection{Propagation with Gevrey cutoff}
In the previous section, we fixed some $\nu$ and could not have extracted dependence of the rank estimates on $\nu$ since we do not know how the constants $C_n$ behave. In particular, the dependence on $\nu$ is buried the choice of $L$ in \cref{eq:L_condition}. However, if we assume that $\chi \in \mathcal{G}^s_{\mathrm{c}}((0, 1))$, we get more explicit control over the decay of $\tilde g$. 

For $N \ge \nu^s \ge 1$, put the time of propagation as
\begin{equation}\label{eq:g_ell_choice}
    \ell = \left \lceil \frac{\log \left( \frac{N}{\nu^s} \right)}{\log M} \right \rceil
\end{equation}
With $\chi \in \mathcal{G}^s_{\mathrm{c}}((0, 1))$ for $s > 1$, the remainder bound (\ref{eq:mathcal_R}) then gives
\begin{equation}\label{eq:g_rem_bound}
    \|\mathcal R_\ell\|_{\ell^2 \to \ell^2} \le C\sum_{j = 0}^{\ell - 1} \lambda^{-j - 1} e^{-c d_{\ell - j}^{1/s}} = \sum_{j = 0}^{\ell - 1} e^{\nu(j + 1)\log M - cd_{\ell - j}^{1/s}}
\end{equation}
where $C$ and $c$ depend only on $\chi$. Again, the remainder bound holds only if $d_j/N$ is sufficiently small for all $j$ according to \cref{r_cond'}. This condition will eventually be fulfilled using the choice of propagation time given by \cref{eq:g_ell_choice} and choosing $N$ to be sufficiently large. First, we need to choose the gap distances. We see from \cref{eq:g_rem_bound} that we should put
\begin{equation}
    d_{\ell - j} ^{1/s} = \Big( \frac{\nu \log M + \mu}{c} \Big) (j + 1),
\end{equation}
where $\mu$ is chosen to be sufficiently large so that 
\[C\sum_{j = 0}^\infty e^{-\mu(j + 1)} \le \frac{1}{2}.\]
Then if $d_{\ell - j}$ satisfies (\ref{r_cond'}), then the estimate \cref{eq:g_rem_bound} gives the desired remainder bound 
\[\|\mathcal R_\ell \|_{\ell^2 \to \ell^2} \le 1/2.\]
Indeed, note that $d_{\ell - j}$ takes its maximum value at $j = \ell - 1$, so (\ref{r_cond'}) is satisfied if
\begin{equation}
    \frac{2d(\supp \chi, 0)}{M} \ge N^{-1} d_1 = N^{-1} \left( \frac{\nu \log M + \mu}{c} \left \lfloor \frac{\log \left( \frac{N}{\nu^s} \right)}{\log M} \right \rfloor \right)^s
\end{equation}
For $1 \le \nu^s < N$, the above is indeed satisfied for all sufficiently large $N$, and the threshold depends only on $\chi$, $M$, and $\epsilon$.
Finally, to estimate the rank of $A$ in \cref{eq:long-time-prop}, we see from \cref{rank_est} that
\begin{align}
    \rank A &\le 2 \frac{N^\delta}{\nu^{s \delta}} \left[\nu^s + 2 \left(\frac{\nu \log M + \mu}{c} \right)^s \sum_{j = 0}^{\ell - 1} \frac{(j + 1)^s}{M^j} \right] \le CN^\delta \nu^{s(1 - \delta)}
\end{align}
Note analysis above yields the following corollary of \cref{prop:long_time}. 
\begin{corollary}\label{cor:long_time_g}
Consider the quantum open maps given by the triple $(M, \mathcal A, \chi)$ where $\chi \in \mathcal{G}^s_{\mathrm{c}}((0, 1))$ for an $s > 1$. Then for all $1 \le \nu^s < N$, there exists a constant $C_{\chi, M, \epsilon}$ such that for all $N = K M > C_{\chi, M, \epsilon}$ where $K \in \N$, there exists operators $A$, $Z(\lambda)$, and $\mathcal R(\lambda)$ on $\ell^2_N$ as in \cref{prop:long_time} that satisfies the identity \cref{eq:long-time-prop}. Furthermore, they satisfy the remainder estimate 
\[\|\mathcal R(\lambda)\|_{\ell^2_N \to \ell^2_N} \le 1/2\]
and the rank bound 
\[\rank A \le CN^\delta \nu^{s(1 - \delta)}\]
where $C$ does not depend on $N$ or $\nu$. 
\end{corollary}

\section{Weyl bounds}
Now we proceed to bounding the number of eigenvalues in $\Omega_\nu$ as defined in (\ref{eq:Omega}). To do so, we will eventually pass to Jensen's formula from complex analysis:

\begin{lemma}\label{lem:jensen}
    Let $f(z)$ be a holomorphic function on a connected open set $\Omega \subset \C$. Let $K \subset \Omega$ be a compact subset. Suppose there exists a constant $L > 0$ and a point $z_0 \in K$ such that 
    \begin{equation}
        \sup_{z \in \Omega} |f(z)| \le e^L, \quad |f(z_0)| \ge e^{-L}.
    \end{equation}
    Then the number of zeros of $f(z)$ in $K$ counted with multiplicity is bounded by 
    \begin{equation}
        |\{z \in K: f(z) = 0\}| \le CL
    \end{equation}
    where the constant $C$ depends only on the geometry, i.e. $z_0$, $\Omega$, and $K$. 
\end{lemma}
See \cite[Lemma 4.4]{DZ} for a proof of the lemma. 

We want to apply Lemma \ref{lem:jensen} to some expression involving a factor of $\det(B_N - \lambda)$ in the region $\Omega_\nu$ (defined in (\ref{eq:Omega})) in order to count the number of eigenvalues in $\Omega_\nu$.  To get the lower bound at a point in $\Omega_\nu$ required in Lemma \ref{lem:jensen}, we first modify the approximate inverse identity (\ref{eq:long-time-prop}) as follows:
\begin{align}
    I &= \lambda^{-1} B_N - \lambda^{-1} (B_N - \lambda) \nonumber \\
    &= Z(B_N - \lambda) + \mathcal R + \lambda^{-1} A B_N - \lambda^{-1} A(B_N - \lambda) \nonumber \\
    &= (Z - \lambda^{-1} A)(B_N - \lambda) + \mathcal R + \lambda^{-1} AB_N \label{eq:forward_back}
\end{align}
where $Z$ and $\mathcal R$ depend holomorphically on $\lambda$. In either the general cutoff setting and the Gevrey cutoff setting, Corollary \ref{cor:long_time} and Corollary \ref{cor:long_time_g} both give the bound $\|\mathcal R(\lambda)\|_{\ell^2_N \to \ell^2_N} \le 1/2$. Therefore, we can define 
\begin{equation}\label{eq:det_express}
\begin{gathered}
    \mathcal B_N(\lambda) := \lambda^{-1} A B_N (I - \mathcal R)^{-1} \\
    F(\lambda) := \det(I - \mathcal B_N)
\end{gathered}
\end{equation}
Note that $F(\lambda)$ is holomorphic in the annulus $\Omega_\nu$. From \cref{eq:forward_back}, we have
\begin{align*}
    F(\lambda) &= \det(I - \mathcal R) \det(I - \mathcal R -  \lambda^{-1} A B_N) \\
    &= \det(I - \mathcal R) \det(Z - \lambda^{-1} Z) \det(B_N - \lambda).
\end{align*}
Therefore if $\lambda$ is an eigenvalue, it must also be a zero of $F(\lambda)$ considered with multiplicity. Thus it suffices to bound the number of zeros of $F$.

\subsection{Proof of Theorem 1}\label{sec:weyl_law_pf}
By Corollary \ref{cor:long_time}, we see that for all $\lambda \in \Omega_\nu$, 
\begin{equation}\label{eq:jen_up}
    |F(\lambda)| \le (\|\mathcal B_N\|_{\ell^2_N \to \ell^2_N} + 1)^{\rank \mathcal B_N} \le (2M^\nu + 1)^{\rank A} \le e^{C N^\delta}
\end{equation}
where the constant $C$ does not depend on $N$. Now we want to find a lower bound on $F(\lambda)$ at a single point. Observe that at $\lambda = 4$,
\[\|\mathcal B_N(4)\|_{\ell^2_N \to \ell^2_N} \le \frac{1}{2},\]
and so
\begin{align}
    |F(4)|^{-1} &= |\det((I - \mathcal B_N(4))^{-1})| \nonumber \\
    &= |\det(I + \mathcal B_N(4)(I - \mathcal B_N(4))^{-1})| \nonumber \\
    &\le \|I + \mathcal B_N(4)(I - \mathcal B_N(4))^{-1}\|_{\ell^2_N \to \ell^2_N}^{\rank \mathcal B_N} \nonumber \\
    &\le e^{CN^\delta} \label{eq:jen_low}
\end{align}
where again the constant $C$ does not depend on $N$. Therefore, Theorem 1 follows from applying Lemma \ref{lem:jensen} to \cref{eq:jen_low} and \cref{eq:jen_up}.

\subsection{Proof of Theorem 2}
We modify the definition of the domain $\Omega_\nu$ slightly to ensure the geometry scales correctly later. Take
\[\Omega_\nu = \{z : M^{-\nu} \le |z| \le e^{5 \nu}\}.\]
Clearly, counting zeros of $F(\lambda)$ for $\lambda \in \Omega_\nu$ suffices, and as long as $e^{5 \nu} \ge 4$, we will be able to find a lower bound at a single point of $F(\lambda)$. 

By Corollary \ref{cor:long_time_g}, for all $|\lambda| \ge M^{-\nu}$,
\begin{equation}\label{eq:G_upper}
    |F(\lambda)| \le (2M^\nu + 1)^{\rank A} \le M^{C \nu \cdot N^{\delta} \nu^{s(1 - \delta)}}.
\end{equation}
where the constant $C$ is independent of $N$ and $\nu$. For a lower bound on large $\lambda$, note that for all $\lambda \ge 4$,
\begin{equation*}
    \|\mathcal B(\lambda) \|_{\ell^2 \to \ell^2} \le \frac{1}{2},
\end{equation*}
and thus for such $\lambda$,
\begin{align}
    |F(\lambda)|^{-1} &= |\det((I - \mathcal B(\lambda))^{-1})| \nonumber \\
    &= |\det(I + \mathcal B(\lambda)(I - \mathcal B(\lambda))^{-1})| \nonumber \\
    &\le M^{C \nu \cdot N^{\delta} \nu^{s(1 - \delta)}}, \label{eq:G_lower}
\end{align}
where again the constant does not depend on $N$ or $\nu$. The domain $\Omega_\nu$ in which we wish to upper bound the number of zeros varies with $\nu$, and the constant in Lemma \ref{lem:jensen} depends on the geometry of the domain. Therefore, in order to capture the dependence on $\nu$, consider the function 
\begin{equation}
    \tilde F(\omega) := F(e^\omega).
\end{equation}
In particular, the number of zeros of $F(\lambda)$ for $\lambda \in \Omega_\nu$ is the same as the number of zeros of $\tilde G(\omega)$ for 
\[\omega \in \{a + bi: a \in [-\nu \log M, 5 \nu], \, b \in [2 \pi k, 2 \pi(k + 1))\}\]
for any $k \in \Z$. Let 
\[\tilde \Omega_\nu = \{a + bi: a \in [\nu \log M, 5\nu], b \in [0, 2 \pi \nu)\}\]
Let $\mathcal N(\nu)$ denote the number of zeros of $F(\lambda)$ for $\lambda \in \Omega_\nu$ and let $\tilde {\mathcal N}(\nu)$ denote the number of zeros of $\tilde F(\omega)$ for $\omega \in \tilde \Omega_\nu$. Note that $\frac{1}{\nu}\tilde \Omega_\nu$ is the same domain independent of $\nu$. Then applying Lemma \ref{lem:jensen} with \cref{eq:G_upper} and \cref{eq:G_lower}, where the latter is taken at the point $e^{4 \nu}$, then for all sufficiently large $\nu$, we have the upper bound
\begin{align*}
    \mathcal N(\nu) &\le \frac{C}{\nu} \tilde{\mathcal{N}}(\nu) \\
    &= \frac{C}{\nu}|\{\text{zeros of $\tilde F(\nu\omega)$ for $\omega \in \frac{1}{\nu}\Omega_\nu$}\}| \\
    &\le CN^\delta \nu^{s(1 - \delta)},
\end{align*}
where the constant $C$ does not depend on $N$ or $\nu$. This concludes the proof.

\section{Numerical discussion}
In this section, we look at how the Weyl upper bounds derived in this paper perform against numerical data. All plots were made using MATLAB, version R2021b. 

We use the same smooth cutoff function as in \cite{DZ} and observe that it is $2$-Gevrey. The cutoff is constructed as follows. Let 
\[f(x) = \int_{-\infty}^{1.02 \cdot x - 0.01} \exp\left(-\frac{1}{t(1 - t)} \right)\, dt.\]
Note that $f(x) = 0$ for $x \le \frac{0.01}{1.02}$ and $f(x) = 1$ for $x \ge \frac{1.01}{1.02}$. Given a tightness parameter $\tau \in (0, 1/2]$, we then define the cutoff
\[\chi = f\left( \frac{x}{\tau} \right) f \left( \frac{1 - x}{\tau} \right).\]
$\chi$ is $2$-Gevrey and is identically $1$ near the interval $[\tau, 1 - \tau]$. 

The MATLAB function \texttt{eig()} was used to compute eigenvalues. We note that column $j$ of $B_{N, \chi}$ is identically zero if $\lfloor j \cdot M/N \rfloor \in \mathcal A$. We cut these columns as well as the corresponding rows from the matrix $B_{N, \chi}$ to form an $K|\mathcal A| \times K |\mathcal A|$ matrix $\tilde B_{N, \chi}$ and compute the eigenvalues of the trimmed matrix using MATLAB. The nonzero eigenvalues of $B_{N, \chi}$ are identical to those of $\tilde B_{N, \chi}$, so for the sake of counting the number of eigenvalues greater than $M^{-\nu}$, using the trimmed matrix only speeds up the computation. 

\begin{figure}
\begin{subfigure}{0.49\textwidth}
  \centering
  \includegraphics[width=\linewidth]{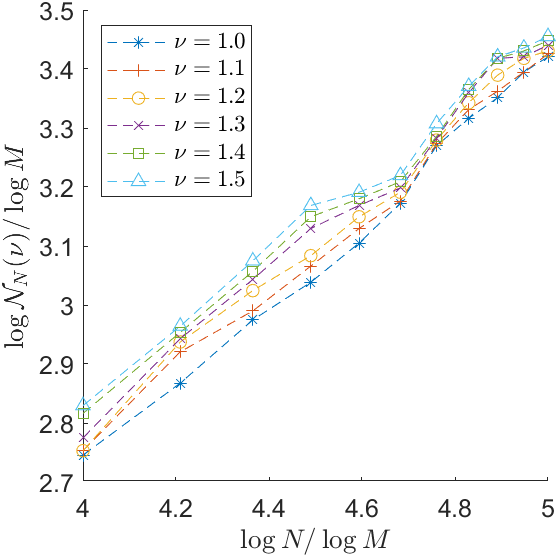}
\end{subfigure}
\hfill
\begin{subfigure}{0.49\textwidth}
  \centering
  \includegraphics[width=\linewidth]{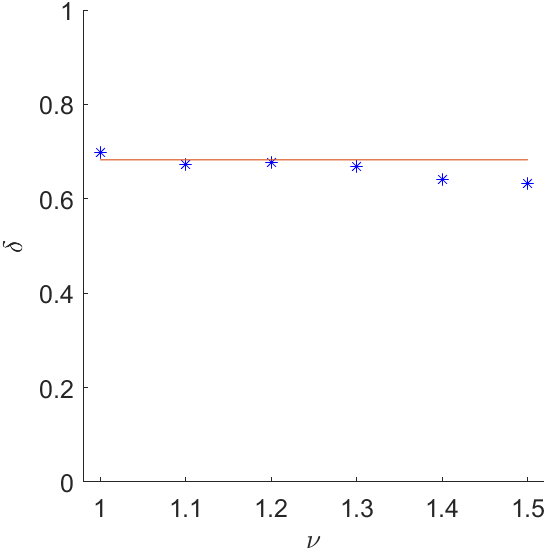}
\end{subfigure}
\caption{$M = 5$, $\mathcal A = \{1, 2, 3\}$, $\tau = 0.05$. Left: $\log \mathcal N_N(\nu)/\log M$ is plotted against $\log N/\log M$ for $K = N/M = \{125, 175, 225, \dots, 625\}$ given various fixed values of $\nu$. From top to bottom, the corresponding values of $\nu$ are $1.5, 1.4, \dots, 1.0$. Right: the corresponding slope of the linear regression of each curve is plotted against $\nu$, and the red line is at $\delta = \log{|\mathcal A|}/\log M$.}
\label{fig:weyl_law}
\end{figure}

\subsection{Dependence on \texorpdfstring{$N$}{}}
For a fixed $\nu$, the counting function $\mathcal N_N(\nu)$ is asymptotically upper bounded by $N^\delta$ as $N = K M \to \infty$. For the numerical experiment in \cref{fig:weyl_law} we plot $\log \mathcal N_N(\nu)/\log M$ against $\log N/\log M$ for several different values of $\nu$, and for each $\nu$, we compute the slope of the linear regression. The numerically computed slopes are all fairly close to $\delta = \frac{\log |\mathcal A|}{\log M}$. Similar numerical results can be obtained for other quantum open baker's maps. This is in numerical agreement with the upper bound derived in this paper, and suggests that there could be matching lower bound, although no such bounds are known. 

The numerics depicted in \cref{fig:weyl_law} is fairly stable under perturbations on the order $10^{-5}$ in the given range of $K$ and $\nu$. In particular, for each $N = K \cdot M$, we also computed the spectrum of $\tilde B_{N, \chi} + P$ where $P$ is a random matrix whose entries are i.i.d. random Gaussians, and the whole matrix is normalized so that $\|P\|_{\ell^2 \to \ell^2} = 10^{-5}$. Running the same experiment as in \cref{fig:weyl_law} with each of the matrices perturbed by a random matrix of norm $10^{-5}$, the differences in the resulting slopes are on the order $10^{-3}$, which suggests a lack of strong pseudo spectral effects in the range of $N$ and $\nu$ of concern.

\begin{figure}
    \centering
    \includegraphics[width = 0.8 \textwidth]{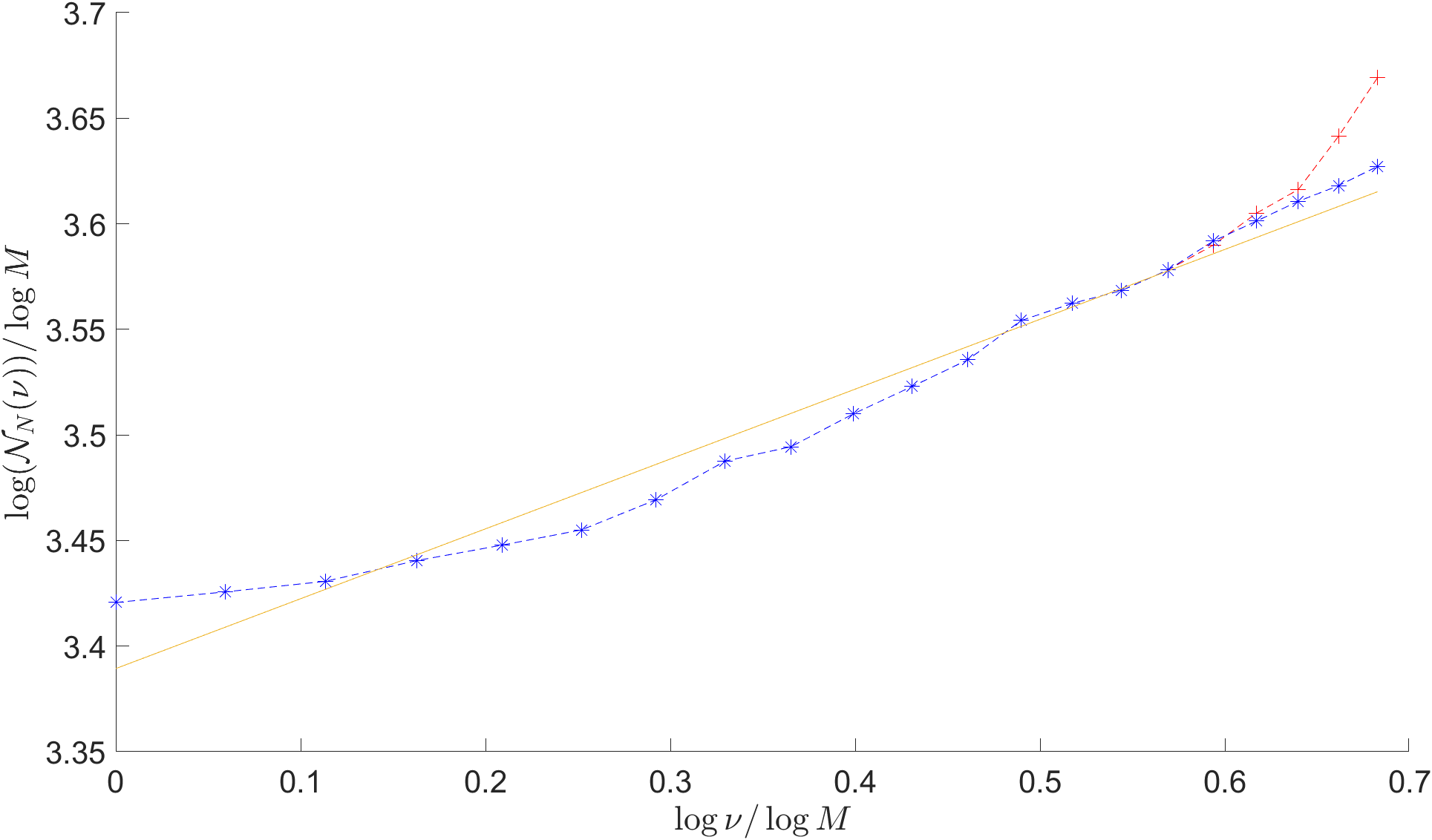}
    \caption{$M = 5$, $\mathcal A = \{1, 2, 3\}$, $\tau = 0.05$, $N = 5^5$. The blue data set is $\log \mathcal N_N(\nu)$ plotted against $\log \nu$ for $\nu = 1.0, 1.1, 1.2, \dots, 3.0$. The red line is the same experiment for the perturbed matrix $\tilde B_{N, \chi} + P$ where $P$ is a randomly chosen matrix normalized so that $\|P\|_{\ell^2 \to \ell^2} \le 10^{-10}$. The yellow line is the linear regression for the experiment in blue, and has a slope of $0.3308$, which is fairly close to $1 - \delta \approx 0.3174$.}
    \label{fig:g_weyl_law}
\end{figure}

\subsection{Dependence on \texorpdfstring{$\nu$}{}} \label{numerics}
Now we fix a large $N = K M$ and see how the counting function $\mathcal N_N(\nu)$ varies with $\nu$. Since we have the asymptotic upper bound $\sim \nu^{s(1 - \delta)}$, we plot $\log \nu/ \log M$ against $\log \mathcal N_N(\nu) / \log M$. The numerical data for $\mathcal N_N(\nu)$ becomes more unstable as $\nu$ becomes large. In \cref{fig:g_weyl_law}, we fix $N = 5^5$ and go to largest $\nu$ for which a perturbation on the order $10^{-10}$ yields no discernible difference. The line of best fit as depicted in \cref{fig:g_weyl_law} has a slope of $0.3308$, which is much closer to $1 - \delta \approx 0.3174$ than to $s(1 - \delta)$. We note that the alphabet and $\chi$ we chose is such that $\chi$ is identically 1 on the Cantor set associated with the alphabet. In fact, similar experiments with the cutoff identically one near the Cantor set has similar behavior in that $\mathcal N_N(\nu)$ behaves like $\nu^{1 - \delta}$. However, the Weyl bound $\sim N^\delta \delta^{s(1 - \delta)}$ holds for all choices of alphabets and $s$-Gevrey cutoffs.

In the edge case where we take the alphabet to be a single point and $\delta = 0$, we should see that the magnitude of the first few eigenvalues to decrease exponentially, which would exhibit the $\mathcal N_N(\nu) \lesssim \nu^2$ behavior. Indeed, this is what we see in \cref{fig:delta_0}. We remark that essentially the same values for the top eigenvalues is obtained if we take other values of $N$, which makes sense in light of the fact that the upper bound we derived is independent of $N$ for the degenerate case $\delta = 0$. 

\begin{figure}
\begin{subfigure}{0.49\textwidth}
  \centering
  \includegraphics[width=\linewidth]{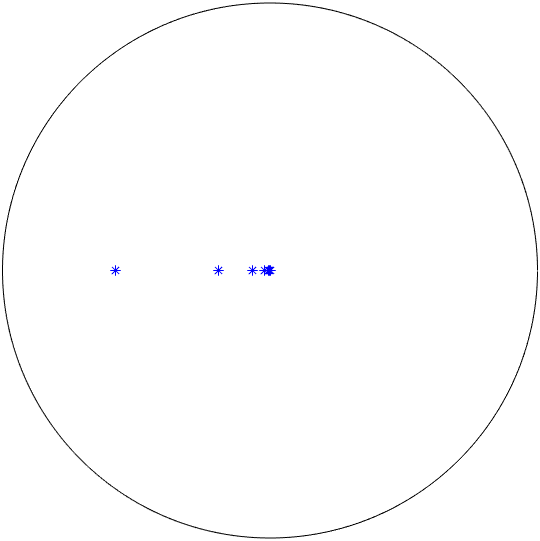}
\end{subfigure}
\hfill
\begin{subfigure}{0.49\textwidth}
  \centering
  \includegraphics[width=\linewidth]{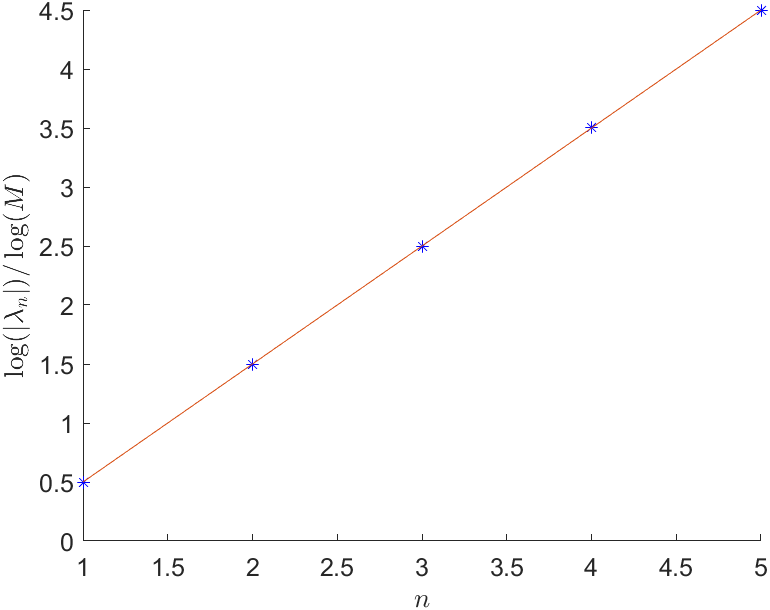}
\end{subfigure}
\caption{$M = 3$, $\mathcal A = \{1\}$, $\tau = 0.1$, $N = 3^6$. In particular, this is the degenerate case where the dimension of the trapped set is $\delta = 0$. Left: A plot of the eigenvalues in the complex unit circle. Right: Log of the magnitude of the five largest eigenvalues}
\label{fig:delta_0}
\end{figure}

\medskip\noindent\textbf{Acknowledgements.}
The author would like to thank Semyon Dyatlov for suggesting this project, for all the the insightful discussions about the methods in his paper with Long Jin \cite{DZ} on which the current paper is based, and for the helpful comments and suggestions on early drafts of this paper.

\bibliographystyle{alpha}
\bibliography{oqm.bib}

\end{document}